\documentclass[11pt]{amsart}

\usepackage{epsfig,amsmath,amsfonts,latexsym}
\usepackage{color}
\usepackage{overpic}
 \usepackage{palatino}
 \usepackage{hyperref}
 \usepackage[nocompress]{cite}
 \usepackage{tikz}
\usetikzlibrary{calc,decorations.pathmorphing,shapes}

\newcounter{sarrow}

\newtheorem{theorem}{Theorem}
\newtheorem{proposition}[theorem]{Proposition}

\newtheorem{conjecture}[theorem]{Conjecture}

\theoremstyle{definition}

\newtheorem{remark}[theorem]{Remark}
\newtheorem{example}[theorem]{Example}

\theoremstyle{remark}

\def\R{\mathbb{R}}

\theoremstyle{plain}

\begin{document}

\title{Homology spheres bounding acyclic smooth manifolds and symplectic fillings}

\author{John B. Etnyre}

\author{B\"{u}lent Tosun}

\address{Department of Mathematics \\ Georgia Institute of Technology \\ Atlanta \\ Georgia}

\email{etnyre@math.gatech.edu}

\address{Department of Mathematics\\ University of Alabama\\Tuscaloosa\\Alabama}

\email{btosun@ua.edu}


\begin{abstract}
In this paper, we collect various structural results to determine when an integral homology $3$--sphere bounds an acyclic smooth $4$--manifold, and when this can be upgraded to a Stein manifold. In a different direction we study whether smooth embedding of connected sums of lens spaces in $\mathbb{C}^2$ can be upgraded to a Stein embedding, and determined that this never happens.    
\end{abstract}

\maketitle

\section{Introduction}
The problem of embedding one manifold into another has a long, rich history, and proved to be tremendously important for answering various geometric and topological problems. The starting point is the Whitney Embedding Theorem: every compact $n$--dimensional manifold can be smoothly embedded in $\mathbb {R}^{2n}$. %Note that it can be arranged so that a manifold of dimension less than $n$ embeds into the Euclidean space $\mathbb{R}^n$ if and only if it embeds into the $n$-dimensional sphere $S^n$.   

In this paper we will focus on smooth embeddings of $3$--manifolds into $\R^4$ and embeddings that bound a convex symplectic domain in $(\R^4, \omega_{std})$. One easily sees that if a (rational) homology sphere has such an embedding, then it must bound a (rational) homology ball. Thus much of the paper is focused on constructing or obstructing such homology balls. 

\subsection{Smooth embeddings}

In this setting, an improvement on the Whitney Embedding Theorem, due to Hirsch \cite{Hirsch:embedding} (also see Rokhlin \cite{Rohlin:3manembedding} and Wall \cite{Wall:embedding}), proves that every $3$--manifold embeds in $\mathbb{R}^5$ smoothly. In the smooth category this is the optimal result that works for all $3$--manifolds; for example, it follows from work of Rokhlin that the Poincar\'{e} homology sphere $P$ cannot be embeded in $\mathbb{R}^4$ smoothly.  On the other hand in the topological category one can always find embeddings into $\R^4$ for any integral homology sphere by Freedman's work \cite{Freedman:4manifolds}. Combining the works of Rokhlin \cite{Rohlin:invariant} and Freedman  \cite{Freedman:4manifolds} for $P$ yields an important phenomena in $4$--manifold topology: there exists a closed oriented non-smoothable $4$--manifold --- the so called $E_8$ manifold. In other words, the question of {\it when does a $3$--manifold embeds in $\mathbb{R}^4$ smoothly} is an important question from the point of smooth $4$--manifold topology. This is indeed one of the question in the Kirby's problem list \cite[Problem~$3.20$]{Kirby:problemlist}.  Since the seminal work of Rokhlin in 1952, there has been a great deal of progress towards understanding this question. On the constructive side, Casson-Harrer \cite{CH}, Stern, and Fickle \cite{Fickle} have found many infinite families of integral homology spheres that embeds in $\mathbb{R}^4$.  On the other hand techniques and invariants, mainly springing from Floer and gauge theories, and symplectic geometry \cite{FS:R, Manolescu:T,OS:grading}, have been developed to obstruct smooth embeddings of $3$--manifolds into $\mathbb{R}^4$. It is fair to say that despite these advances and lots of work done in the last seven decades, it is still unclear, for example, which Brieskorn homology spheres embed in $\mathbb{R}^4$ smoothly and which do not.   

A weaker question is whether an integral homology sphere can arise as the boundary of an acyclic $4$--manifold. A manifold that has all its reduced homology trivial is called acyclic. So (rationally) acyclic 4-manifold is the same as (rational) homology ball.  Note that an integral homology sphere that embeds in $\mathbb{R}^4$ necessarily bounds an integral homology ball, and hence is  homology cobordant to the $3$--sphere. Thus a homology cobordism invariant could help to find restrictions, and plenty of such powerful invariants have been developed. For example, for odd $n$, $\Sigma(2,3,6n-1)$ and $\Sigma(2,3,6n+1)$ have non-vanishing Rokhlin invariant. For even $n$, $\Sigma(2,3,6n-1)$ has $R=1$, where $R$ is the invariant of Fintushel and Stern, \cite{FS:R}. Hence none of these families of homology spheres can arise as the boundary of an acyclic manifold. On the other hand, for $\Sigma(2,3,12k+1)$ all the known homology cobordism invariants vanish. Indeed, it is known that $\Sigma(2,3,13)$ \cite{AKi} and $\Sigma(2,3,25)$ \cite{Fickle} bound contractible manifolds of Mazur type. Motivated by the questions and progress mentioned above and view towards their symplectic analogue, we would like to consider some particular constructions of $3$--manifolds bounding acyclic manifolds.        

Our first result is the following, which follows by adapting a method of Fickle. 

\begin{theorem}\label{main1}
Let $K$ be a knot in the boundary of an acyclic, respectively rationally acyclic,  $4$--manifold $W$ which has a genus one Seifert surface $F$ with a primitive element $[b]\in H_1(F)$ such that the curve $b$ is slice in $W.$ If $b$ has $F$--framing $s$, then the homology sphere obtained by $\frac{1}{(s\pm 1)}$ Dehn surgery on $K$ bounds an acyclic, respectively rationally acyclic, $4$--manifold. 
\end{theorem} 
\begin{remark}
Notice that the normal bundle to the slice disk for $b$ has a unique trivialization and thus frames $b$.  The $F$--framing on $b$ is simply the difference between the framing of $b$ given by $F$ and the one given by the slicing disk.  
\end{remark}
\begin{remark}
Fickle \cite{Fickle} proved this theorem under the assumption that $\partial W$ was $S^3$ and $b$ was an unknot, but under these stronger hypothesis he was able to conclude that the homology sphere bounds a contractible manifold. 
\end{remark}
\begin{remark}
Fintushel and Stern conjectured, see \cite{Fickle}, the above theorem for $\frac{1}{k(s\pm 1)}$ Dehn surgery on $K$, for any $k\geq 0$. So the above theorem can be seen to verify their conjecture in the $k=1$ case. 

As noted by Fickle, if the conjecture of Fintushel and Stern is true then $\Sigma(2,3,12k+1)$ will bound an acyclic manifold for each $k\geq 1$ since they can be realized by $-1/2k$  surgery on the right handed trefoil knot, and this knot bounds a Seifert surface containing an unknot to which the surface gives framing $-1$.
\end{remark}
\begin{remark}\label{allowribbon}
Notice that if $b$ is as in the theorem, then the Seifert surface $F$ can be thought of as obtained by taking a disk around a point on $b$, attaching a 1--handle along $b$ (twisting $s$ times) and then attaching another 1--handle $h$ along some other curve. The proof of Theorem~\ref{main1} will clearly show that $F$ does not have to be embedded, but just ribbon immersed so that cutting $h$ along a co-core to the handle will result in a surface that is ``ribbon isotopic" to an annulus. By ribbon isotopic, we mean there is a 1-parameter family of ribbon immersions between the two surfaces, where we also allow a ribbon immersion to have isolated tangencies between the boundary of the surface and an interior point of the surface. (In the proof we will see that it is important that the handle attached along $b$ does not pierce the rest of the surface. We can only allow the handle $h$ to pierce the surface.)
\end{remark}

\begin{example}
Consider the $n$-twisted $\pm$ Whitehead double $W_n^\pm(K_p)$ of $K_p$ from Figure~\ref{exampleKp} (here the $\pm$ indicates the sign of the clasp in the double). In \cite{Cha07}, Cha showed that $K_p$ is rationally slice. That is $K_p$ bounds a slice disk in some rational homology $B^4$ with boundary $S^3$. (Notice that $K_1$ is the figure eight knot originally shown to be rationally slice by Fintushel and Stern \cite{FintushelStern84}.) Thus Theorem~\ref{main1} shows that $\frac{1}{n\pm 1}$ surgery on $W_n^\pm(K_p)$ bounds a rationally acyclic 4--manifold. This is easy to see as a Seifert surface for $W_n^\pm(K_p)$ can be made by taking an $n$-twisting ribbon along $K_p$ and plumbing a $\pm$ Hopf band to it. JungHwan Park noted that when $n=0$ these knots are rationally concordant to the unknot and hence $\pm1$ surgery on $W_0^\pm(K_p)$ bounds a rational homology ball. 
\begin{figure}[htb]{\tiny
\begin{overpic}%[grid,tics=10] 
{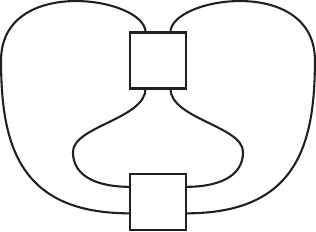}
\put(44.5, 51){\Large$-p$}
\put(48,6){\Large$p$}
\end{overpic}}
\caption{The rationally slice knot $K_p$.}
\label{exampleKp}
\end{figure}

Moreover, from Fickle's original version of the theorem, $\pm\frac12$ surgery on $W_0^\pm(K_p)$ bounds a contractible manifold. 

We can generalize this example as follows. Given a knot $K$, we denote by $R_m(K)$ the $m$-twisted ribbon of $K$. That is take an annulus with core $K$ such that its boundary components link $m$ times. We can now plumb  $R_{m_1}(K_1)$ and $R_{m_2}(K_2)$ by identifying a neighborhood $N_i$ of a point on $K_i$ in $R_{m_i}(K_i)$ with $[-1,1]\times [-1,1]$ so that $[-1,1]\times \{0\}$ is a neighborhood of the point on $K_i$, and then gluing $N_1$ to $N_2$ by interchanging the interval factors. Denote the boundary of this new surface by $P(K_1, K_2, m_1, m_2)$. If the $K_i$ are rationally slice then $\frac{1}{m_i\pm 1}$ surgery on $P(K_1, K_2, m_1, m_2)$ yields a manifold bounding a rationally acyclic manifold; moreover, if the $K_i$ are slice in some acyclic manifold, then the result of these surgeries will bound an acyclic manifold. 

The knots in Figure~\ref{exampleKp} give a good source for the $K_i$ in the construction above. Another good source comes from a result of Kawauchi \cite{Kawauchi09}. He showed that if $K$ is a strongly negative-amphichiral knot (meaning there is an involution of $S^3$ that takes $K$ to its mirror with reversed orientation), then the $(2m,1)$--cable of $K$ is rationally slice for any $m\not=0$. We can apply this to the knots in Figure~\ref{exampleKp} to obtain another family of rationally slice knots. 
\end{example}

{\bf Symplectic embeddings.} Another way to build examples of integral homology spheres that bound contractible manifolds is via the following construction. Let $K$ be a slice knot in the boundary of a contractible manifold $W$ (e.g. $W=B^4$), then $\frac{1}{m}$ Dehn surgery along $K$ bounds a contractible manifold. This is easily seen by removing a neighborhood of the slice disk from W (yielding a manifold with boundary the $0$ surgery on $K$) and attaching a 2--handle to a meridian of $K$ with framing $-m$. With this construction one can find examples of  $3$--manifolds modeled on not just Seifert geometry, for example $\Sigma(2,3,13)$ is the result of $1$  surgery on Stevedore's knot $6_1$, but also hyperbolic geometry, for example the boundary of the Mazur cork is the result of $1$  surgery on the pretzel knot $P(3,-3,-3)$, which is also known as $m({9}_{46})$. 
See Figure~\ref{fig:smooth} (we use the standard conventions to frame knots that run over $1$--handles, see \cite[Section~2]{Gompf}).
\begin{figure}[h!]
\begin{center}
  \includegraphics[width=16cm]{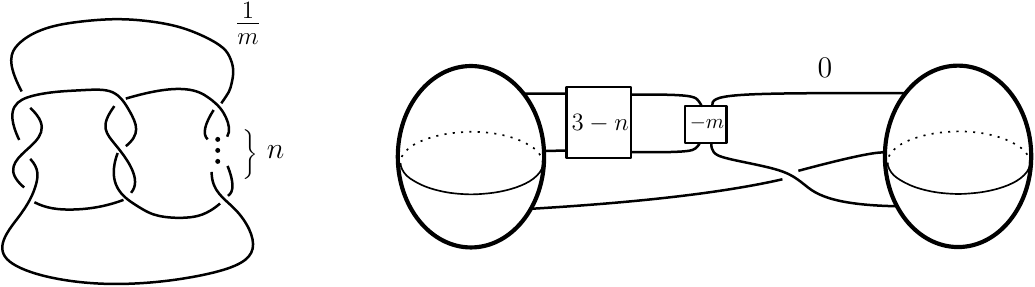}
 \caption{On the left is the 3-manifold $Y_{m,n}$ described as a smooth $\frac{1}{m}$ surgery on the slice knot $P(3,-3,-n)$ for $n\geq 3$ odd. On the right is the contractible Mazur-type manifold $W_{m,n}$ with $\partial W_{m,n}\cong Y_{m,n}$. Note the $m=1,~n=3$ case yields the original Mazur manifolds (with reversed orientation).}
  \label{fig:smooth}
\end{center}
\end{figure}

We ask the question of when $\frac{1}{m}$ surgery on a slice knot produces a contact $3$--manifold that bounds a {\it Stein contractible} manifold. Here there is an interesting asymmetry not seen in the smooth case. 

\begin{theorem}\label{regslice}
 Let $L$ be a Legendrian knot in $(S^3, \xi_{std})$ that bounds a regular Lagrangian disc in $(B^4, w_{std})$.  Contact $(1+\frac{1}{m})$  surgery on $L$ (so this is smooth $\frac{1}{m}$  surgery) is the boundary of a contractible Stein manifolds if and only if $m>0$.
\end{theorem}

This result points out an interesting angle on a relevant question in low dimensional contact and symplectic geometry: which compact contractible 4-manifolds admit a Stein structure? In \cite{MT:pseudoconvex} the second author and Mark found the first example of a contractible manifold without Stein structures with either orientation. This manifold is a Mazur-type manifold with boundary the Brieskorn homology sphere $\Sigma(2,3,13)$. A recent conjecture of Gompf remarkably predicts that no non-trivial Brieskorn homology sphere bounds an acyclic Stein manifold. As observed above $\Sigma(2,3,13)$ is the result of smooth $1$ surgery along the stevedore's knot $6_1$. The knot $6_1$ is not Lagrangian slice, and indeed if Gompf conjecture is true, then by Theorem~\ref{regslice} $\Sigma(2,3,13)$ can never be obtained as a smooth $\frac{1}{n}$ surgery on a Lagrangian slice knot for any natural number $n$. Motivated by this example, Theorem~\ref{regslice}, and Gompf's conjecture we make the following weaker conjecture.

\begin{conjecture}
No non-trivial Brieskorn homology sphere $\Sigma(p,q,r)$ can be obtained as smooth $\frac{1}{n}$ surgery on a regular Lagrangian slice knot.
\end{conjecture}     

On the other hand as in Figure ~\ref{fig:smooth} we list a family of slice knots, that are regular Lagrangian slice because they bound decomposable Lagrangian discs and by \cite{CET} decomposable Lagrangian cobordisms/fillings are regular. We explicitly draw a contractible Stein manifold  $X_{m,n}$ these surgeries bound in Figure~\ref{fig:stein}.

\begin{figure}[h!]
\begin{center}
  \includegraphics[width=12cm]{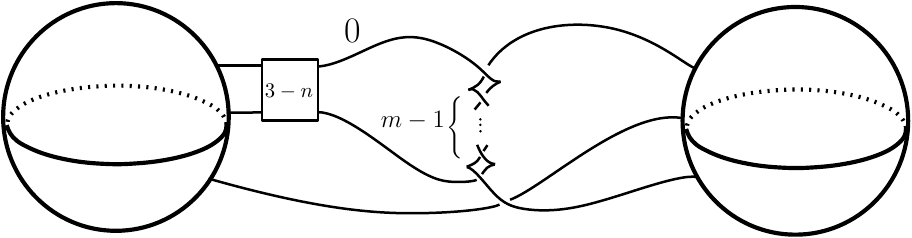}
 \caption{Stein contractible manifold with $\partial X_{m,n}\cong Y_{m,n}$.}
  \label{fig:stein}
\end{center}
\end{figure}

 A related embedding question is the following: when does a lens space $ L(p,q)$ embed in $\mathbb{R}^4$ or $S^4$? Two trivial lens spaces, $S^3$ and $S^1\times S^2$ obviously have such embeddings. On the other hand, Hantzsche in $1938$ \cite{Hantzsche} proved, by using some elementary algebraic topology that if a $3$--manifold $Y$ embeds in $S^4$, then the torsion part of $H_1(Y)$ must be of the form $G \oplus G$ for some finite abelian group $G$. Therefore a lens space $L(p,q)$ for $|p|>1$ never embeds in  $S^4$ or $\mathbb{R}^4$. For punctured lens spaces, however the situation is different. By combining the works of Epstein \cite{Epstein} and Zeeman \cite{Zeeman}, we know that, a punctured lens space $L(p,q)\setminus B^3$ embeds in $\mathbb{R}^4$ if and only if $p>1$ is odd. Note that given such an embedding a neighborhood of $L(p,q)\setminus B^3$ in $\mathbb{R}^4$ is simply $(L(p,q)\setminus B^3)\times [-1,1]$ a rational homology ball with boundary $L(p,q)\# L(p,p-q)$ (recall $-L(p,q)$ is orientation-preserving diffeomorphic to $L(p,p-q)$). 
 
 One way to see an embedding of $L(p,q)\# L(p,p-q)$ into $S^4$ is as follows: First, it is an easy observation that if $K$ is a doubly slice knot (that is there exists a smooth unknotted sphere $S\subset S^4$ such that $S\cap S^3=K$), then its double branched cover $\Sigma_2(K)$ embeds in $S^4$ smoothly. Moreover by a known result of Zeeman $K \#{-m(K)}$ is a doubly slice knot for any knot $K$ (here $-m(K)$ is the mirror of $K$ with the reversed orientation).  It is a classic fact that $L(p,q)$ is the double branched cover over the 2-bridge knot $K(p,q)$ (this is exactly where we need $p$ to be odd, as otherwise $K(p,q)$ is a link). In particular $L(p,q)\#L(p,p-q)$, being the double branched cover of the doubly slice knot $K(p,q)\#{-m(K(p,q))}$, embeds in $S^4$ smoothly. On the other hand, Fintushel and Stern \cite{FS:lensspace} showed this is all that could happen. That is they proved that $L(p,q)\#L(p,q')$ embeds in $S^4$ if and only if $L(p,q')=L(p,p-q)$ and $p$ is odd. (Previously Gilmer and Livingston \cite{GilmerLivingston} had shown this when $p$ was a prime power less than 231, though they also obstructed topological embeddings whereas Fintushel and Stern's work concerned smooth embeddings). In particular for $p$ odd, $L(p,q)\#L(p,p-q)$ bounds a rational homology ball in $\mathbb{R}^4$. A natural question in this case is to ask whether any of this smooth rational homology balls can be upgraded to be symplectic or Stein submanifold of $\mathbb{C}^2$. We prove that this is impossible.   

\begin{theorem}\label{main2}
No contact structure on $L(p,q)\#L(p,p-q)$ has a symplectic filling by a rational homology ball, assuming $p>1$ (that is $L(p,q)$ and $L(p,p-1)$ are not $S^3$). In particular, $L(p,q)\#L(p,p-q)$ cannot embed in $\mathbb{C}^2$ as the boundary of exact symplectic submanifold in $\mathbb{C}^2$. 
\end{theorem}

\begin{remark}
Donald \cite{Donald} generalized Fintushel-Stern and Gilmer-Livingston's construction further to show that for $L=\#_{i=1}^h L(p_i,q_i)$, the manifold $L$ embeds smoothly in $\mathbb{R}^4$ if and only if each $p_i$ is odd, and there exists $Y$ such that $L\cong Y\#-Y$. Our proof of Theorem~\ref{main2} applies to this generalization to prove none of the sums of lens spaces which embed in $\mathbb{R}^4$ smoothly can bound an exact symplectic manifold in $\mathbb{C}^2.$  
\end{remark}

To prove this theorem we need a preliminary result of independent interest.
\begin{proposition}\label{prop}
If a symplectic filling $X$ of a lens space $L(p,q)$ is a rational homology ball, then the induced contact structure on $L(p,q)$ is a universally tight contact structure $\xi_{std}$. 
\end{proposition}

\begin{remark}
Recall that every lens space admits a unique contact structure $\xi_{std}$ that is tight when pulled back the covering space $S^3$. Here we are not considering an orientation on $\xi_{std}$ when we say it is unique. On some lens spaces the two orientations on $\xi_{std}$ give the same oriented contact structure and on some they are different. 
\end{remark}

\begin{remark}
After completing a draft of this paper, the authors discovered that this result was previously proven by Golla and Starkston \cite[Proposition~A$.2.$]{GollaStarkston19pre}. Fossati \cite{Fossati19pre} had previously constrained the topology of fillings of virtually overtwisted contacts structures on a restricted class of lens spaces. As the proof we had is considerably different we decided to present it here.  
\end{remark}

%Recall a small Seifert fibered space is a manifold that can be described by the surgery diagram on the left of Figure~\ref{}. We denotes this by $\Sigma(e_0,r_1,r_2,r_3)$ with the $r_i\in(0,1)$ (we exclude the cases where $r_i=1$ as they are lens spaces and covered by the above result). By work of Wu \cite{} we know that if $e_0\leq -3$ then all the contact structures on $\Sigma(e_0,r_1,r_2,r_3)$ come from a Legendrian surgery on some Legendrian realization of the left hand side of Figure~\ref{}, where $[a_0,\ldots, a_{n_1}]$ is a continued fractions expansion of $-1/r_1$ with each $a_i\leq -2$ and similarly $[b_1,\ldots, b_{n_2}]$ and $[c_1,\ldots, c_{n_3}]$ is such an expansion of $-1/r_2$ and $-1/r_3$, respectively. We call the chain of unknots corresponding to one of the continued fractions a {\em chain} of $\Sigma(e_0,r_1,r_2,r_3)$, we also call the central unknot a chain of $\Sigma(e_0,r_1,r_2,r_3)$. We call a Legendrian link realizing a chain {\em consistent} if all the stabilizations are positive or all are negative. 
%
%\begin{proposition}\label{prop2}
%Let $\xi$ be a contact structure on $\Sigma(e_0;r_1,r_2,r_3)$ with $e_0\leq -3$ given by a Legendrian surgery presentation in Figure~\ref{}. Unless each of the chains is consistent, then there is no rational homology ball symplectic filling of $\xi$. 
%\end{proposition}

\medskip
\noindent
{\bf Acknowledgements:} We are grateful to Agniva Roy for pointing out the work of Fossati and of Golla and Starkston. We also thank Marco Golla, JungHwan Park, and Danny Ruberman for helpful comments on the first version of the paper. We are also grateful to the referee for the careful reading of the paper and many helpful suggestions that greatly improved the paper. 
The first author was partially supported by NSF grant DMS-1906414. Part of the article was written during the second author's research stay in Montreal in Fall 2019. This research visit was supported in part by funding from the Simons Foundation and the 
Centre de Recherches Math\'ematiques, through the Simons-CRM scholar-in-residence program. The second author is grateful to CRM and CIRGET, and in particular to Steve Boyer for their wonderful hospitality. The second author was also supported in part by a grant from the Simons Foundation (636841, BT)

%%%%%%%%%%%%%%%%%%%%%%%%%%%%%%%%%%%%%%%%%%%
\section{Bounding acyclic manifolds}
%%%%%%%%%%%%%%%%%%%%%%%%%%%%%%%%%%%%%%%%%%%
We now prove Theorem~\ref{main1}. The proof largely follows Fickle's argument from \cite{Fickle}, but we repeat it here for the readers convince (and to popularize Fickle's beautiful argument) and to note where changes can be made to prove our theorem. 
\begin{proof}[Proof of Theorem~\ref{main1}]
Suppose the manifold $\partial W$ is given by a surgery diagram $D$. Then the knot $K$ can be represented as in Figure~\ref{TheKnot}. %since $b$ represents a primitive element in homology there is another curve $b'$ such that $b$ and $b'$ intersect in a single point and generate the first homology of the surface. The surface $F$ is can be built with a single $0$--handle and two $1$--handles. Up to isotopy we can take the $0$--handle to be a neighborhood of the intersection between $b$ and $b'$ and the $1$--handles to be thin neighborhoods of the parts of $b$ and $b'$ outside the $0$--handle. 
There we see in grey the ribbon surface $F$ with boundary $K$ and the curve $b$ on the surface. One may see this as follows: 
\begin{figure}[htb]{\tiny
\begin{overpic}%[grid,tics=10] 
{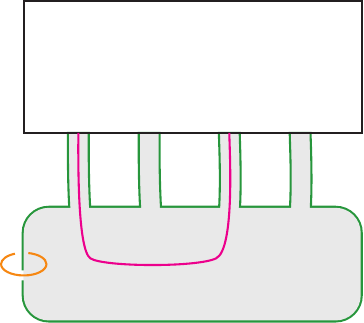}
\put(48, 68){\Large$D$}
\put(103, 15){$K$}
\put(62, 15){$b$}
\put(11, 10){$-s+1$}
\put(90, 34){$0$}
\end{overpic}}
\caption{The knot $K$ bounding the surface $F$ (in grey) in $\partial W$ represented by the diagram $D$. The two 1--handles of $F$ can interact in the box $D$ and have ribbon singularities as described in the Remark~\ref{allowribbon}. The 1--handle neighborhood of $b$ induces framing $s$ on $b$. The framings on $K$ are with respect to the framing coming from the surface $F$.}
\label{TheKnot}
\end{figure}
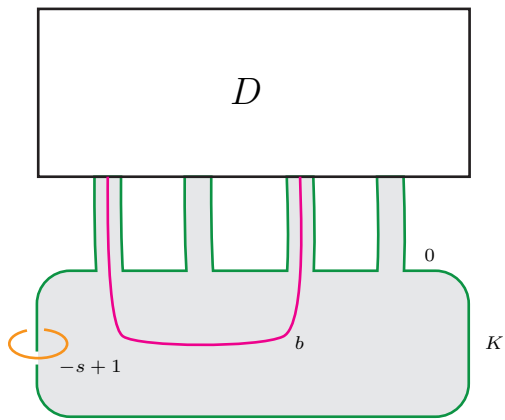
%To see that we can take $b$ to run over a $1$--handle of the surface exactly once, 
notice that since $b$ is primitive in homology, it is a non-separating curve in $F$, and thus there is an arc $a$ that intersects $b$ exactly once (and transversely) and has boundary on $\partial F$. If we cut $F$ along $a$ we see that the resulting surface is an annulus and what is left of $b$ runs from one boundary component to the other. Let $b'$ be the core of this annulus.  We can recover $F$ from this annulus by attaching a $1$--handle to undo the cut along $a$. In particular we see that the surface $F$ is can be built with a single $0$--handle and two $1$--handles. Up to isotopy we can take the $0$--handle to be a neighborhood of the intersection between $b$ and $b'$ and the $1$--handles to be thin neighborhoods of the parts of $b$ and $b'$ outside the $0$--handle. 
%Notice that the surface $F$ deformation retracts onto $b\cup b'$. 
This establishes the claimed picture. 

The result of $\frac{1}{s-1}$ surgery on $K$ is obtained by doing $0$ surgery on $K$ and $(-s+1)$ surgery on a meridian as shown in Figure~\ref{TheKnot}. (The argument for $\frac{1}{s+1}$ surgery is analogous or can be seen by applying the argument below to $-W$ and the mirror of $K$.) %Though note that these two approaches can produce different $4$--manifolds.) 
Now part of $b$ is the core of one of the 1--handles making up $F$. So we can handle slide $b$ and the associated 1--handle over the $(-s+1)$ framed unknot to arrive at the left hand picture in Figure~\ref{Step1}. Then one may isotope the resulting diagram to get to the right hand side of Figure~\ref{Step1}. 
\begin{figure}[htb]{\tiny
\begin{overpic}%[grid,tics=10] 
{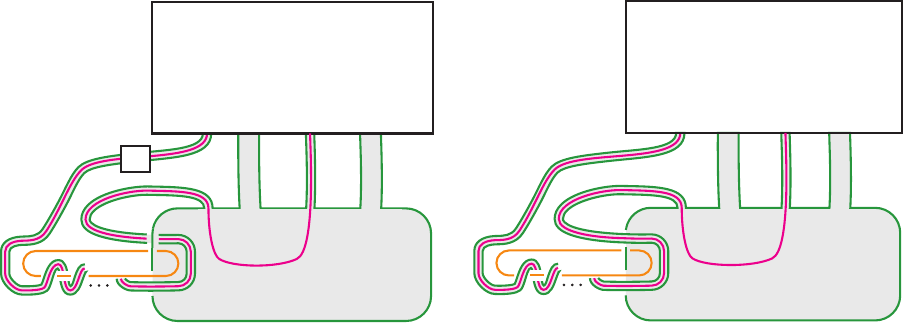}
\put(30, 26){\Large$D$}
\put(84, 26){\Large$D$}
\put(48, 11){$0$}
\put(100, 11){$0$}
%\put(62, 15){$b$}
\put(11, 6){$-s+1$}
\put(63, 6){$-s+1$}
\put(14.5, 17.5){$1$}
\end{overpic}}
\caption{On the left is the result of sliding $b$ and the 1--handle that is a neighborhood hood of $b$ over the $-s+1$ framed unknot. The right hand picture is obtained by an isotopy. (This is an isotopy of the knot, not the surface. The grey surface on the right is not isotopic to the one on the left.)}
\label{Step1}
\end{figure}
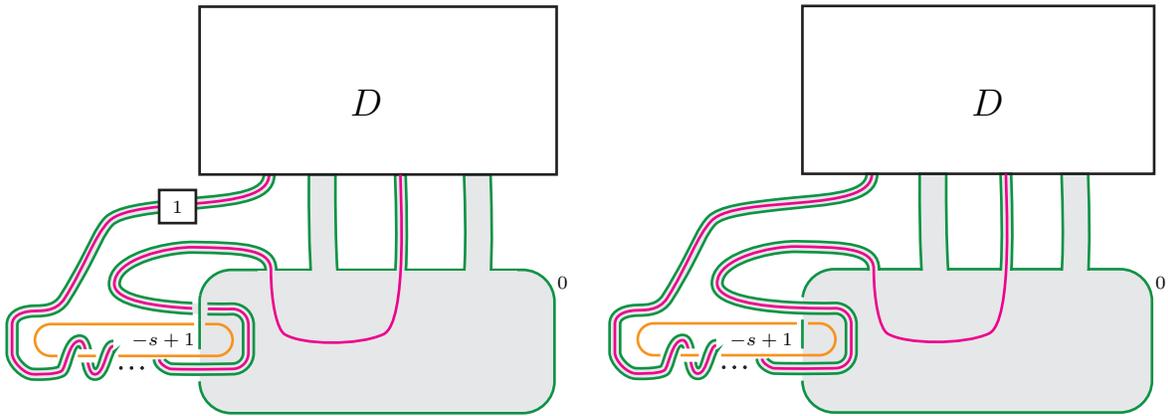
We now claim the left hand picture in Figure~\ref{Step2} is the same manifold as the right hand side of Figure~\ref{Step1}. To see this notice that the green part of the left hand side of Figure~\ref{Step2} consists of two 0-framed knots. Sliding one over the other and using the new 0-framed unknot to cancel the non-slid component results in the right hand side of Figure~\ref{Step1}.

Before moving forward we discuss the strategy of the remainder of the proof. The left hand side of Figure~\ref{Step2} represents the 3-manifold $M$ obtained from $\partial W$ by doing $\frac{1}{s-1}$ surgery on $K$. We will take $[0,1]\times M$ and attach a $0$-framed $2$--handle to $\{1\}\times M$ to get a 4-manifold $X$ with upper boundary $M'$. We will observe that $M'$ is also the boundary of $W$ with a slice disk $D$ for $b$ removed. Since $W$ is acyclic, the complement of $D$ will be a homology $S^1\times D^3$. Let $W'$ denote this manifold. Attaching $X$ upside down to $W'$ (that is attaching a 2--handle to $W'$) to get a 4--manifold $W''$ with boundary $-M$. Since $-M$ is a homology sphere, we can easily see that $W''$ is acyclic. Thus $-W''$ is an acyclic filling of $M$. 

Now to see we can attach the 2--handle to $[0,1]\times M$ as described above, we just add a 0-framed meridian to the new knot on the left hand side of Figure~\ref{Step2}. This will result in the diagram on the right hand side of Figure~\ref{Step2}.
\begin{figure}[htb]{\tiny
\begin{overpic}%[grid,tics=10] 
{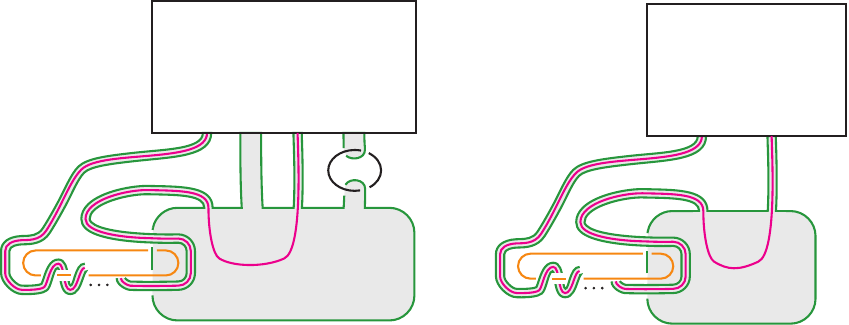}
\put(32, 29){\Large$D$}
\put(87, 29){\Large$D$}
\put(50, 11){$0$}
\put(97, 11){$0$}
\put(45.5, 19){$0$}
\put(26.5, 17.5){$0$}
\put(11, 6.5){$-s+1$}
\put(70, 6){$-s+1$}
\end{overpic}}
\caption{The left hand side describes the same manifold as the right hand side of Figure~\ref{Step1}. The right hand side is the result of attaching a 0-framed 2--handle to the meridian of the new unknot. }
\label{Step2}
\end{figure}
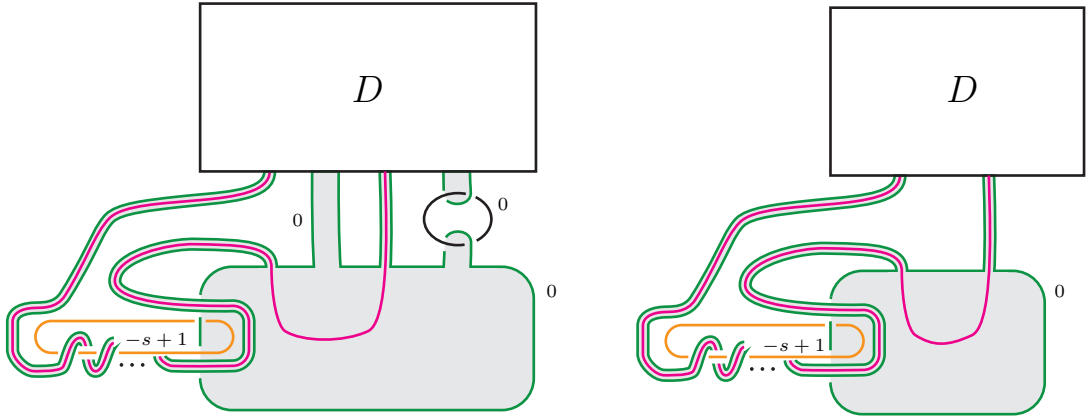

We are left to see that the right hand side of Figure~\ref{Step2} is the boundary of $W$ with the slice disk for $b$ removed. To see this notice that the two green curves in Figure~\ref{Step2} co-bound an embedded annulus with zero twisting (the grey in the figure) and one boundary component links the $(-s+1)$ framed unknot and the other does not. Sliding the former over the latter results in the left hand diagram in Figure~\ref{Step3}. Cancelling the two unknots from the diagram results in the right hand side of Figure~\ref{Step3} which is clearly equivalent to removing the slice disk $D$ for $b$ from $W$. 
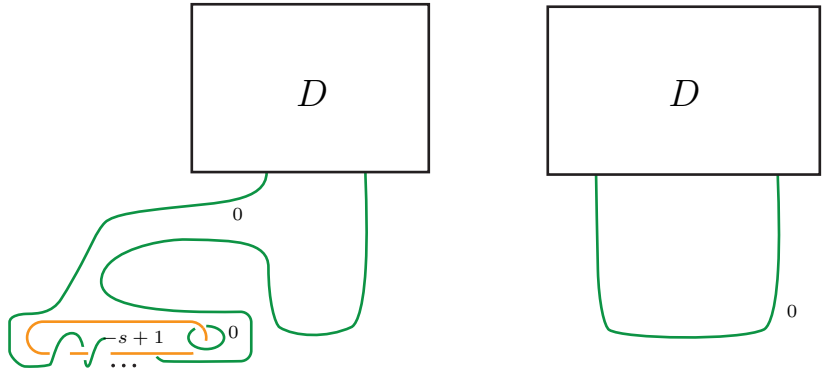
\begin{figure}[htb]{\tiny
\begin{overpic}%[grid,tics=10] 
{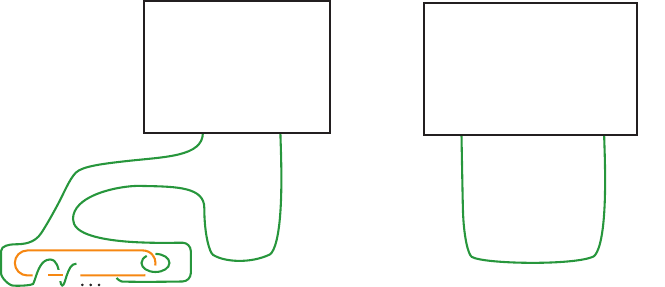}
\put(34, 31){\Large$D$}
\put(78, 31){\Large$D$}
\put(26, 3.5){$0$}
\put(26.5, 17.5){$0$}
\put(11, 3){$-s+1$}
\put(92, 6){$0$}
\end{overpic}}
\caption{The left hand side describes the same manifold as the right hand side of Figure~\ref{Step2}. The right hand side is the result of cancelling the two unknots from the diagram.}
\label{Step3}
\end{figure}
\end{proof}

%%%%%%%%%%%%%%%%%%%%%%%%%%%%%%%%%%%%%%%%%%%
\section{Stein fillings}
%%%%%%%%%%%%%%%%%%%%%%%%%%%%%%%%%%%%%%%%%%%
We begin this section by proving Theorem~\ref{regslice} concerning smooth $\frac 1 m$  surgery on a Lagrangian slice knot. 

\begin{proof}[Proof of Theorem~\ref{regslice}]
We begin by recalling a result from \cite{CET} that says contact $(r)$ surgery on a Legendrian knot $L$ for $r\in(0,1]$ is strongly symplectically fillable if and only if $L$ is Lagrangian slice and $r=1$. Thus $(1+1/m)$ contact surgery for $m<0$ will never be fillable, much less fillable by a contractible Stein manifold. 

We now turn to the $m>0$ case and start by a particularly helpful visualization of the knot $L$, here and below $L$ stands both for the knot type and Legendrian knot that bounds the regular Lagrangian disk $D$ in $B^4$. By \cite[Theorem~$1.9$, Theorem~$1.10$]{CET}, we can find a handle presentation of the $4$--ball $B^4$ made of one 0--handle, and $n$ cancelling Weinstein $1$-- and $2$--handle pairs, and a maximum Thurston-Bennequin unknot in the boundary of the 0--handle that is disjoint from $1$-- and $2$--handles such that when the $1$-- and $2$--handle cancellations are done the unknot becomes $L$. See Figure~\ref{Handle}. 
\begin{figure}[htb]{
\begin{overpic}%[grid,tics=10] 
{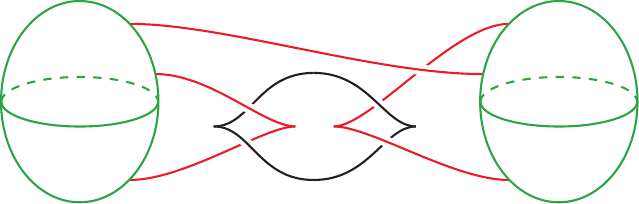}
\put(56, 2){$L$}
\end{overpic}}
\caption{A Stein presentation for the 4--ball together with an ``unknot" labeled $L$. When the cancelling $1$-- and $2$--handles are removed, the knot becomes $L$. In this case $L$ is the pretzel knot $P(3,-3,-3)$.}
\label{Handle}
\end{figure}
Now smooth $1/m$ surgery on $L$ can also be achieved by smooth $0$ surgery (which corresponds to taking the complement of the slice disk $D$) on $L$ followed by smooth $-m$ surgery on its meridian. 

As the proof of Theorem~1.1 in \cite{CET} shows, removing a neighborhood of the Lagrangian disk $D$ that $L$ bounds from $B^4$, gives a Stein manifold with boundary $(+1)$ contact surgery on $L$ (that is smooth $0$ surgery on $L$). Now since the meridian to $L$ can clearly be realized by an unknot with Thurston-Bennequin invariant $-1$, we can stabilize it as necessary and attach a Weinstein $2$--handle to it to get a contractible Stein manifold bounding $(1+1/m)$ contact surgery on $L$ for any $m>1$. 

For the $m=1$ case we must argue differently. One may use Legendrian Reidemeister moves to show that in any diagram for $L$ as described above the $2$--handles pass through $L$ as shown on the left hand side of Figure~\ref{normalform}. 
\begin{figure}[htb]{
\begin{overpic}%[grid,tics=10] 
{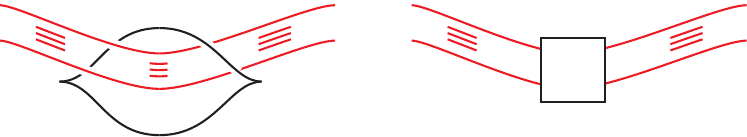}
\put(74, 7.75){$-1$}
\end{overpic}}
\caption{The left hand diagram shows how the $2$--handles in a presentation of $L$ can be normalized. On the right is the result ``blowing down" $L$, that is doing smooth $1$ surgery on $L$ and then smoothly blowing it down. (The box indicates one full left handed twist.)}
\label{normalform}
\end{figure}
To see this consider Figure~\ref{isotonormalform}. There in the upper left we see a general picture for how the attaching spheres for the $2$--handles can run through $L$, here $\mathcal{T}$ is just some Legendrian tangle and the labeling $a, b, c,$ and $d$ mean there are $a$ strands entering from the top left, and similarly for the other labels. The next two diagrams in Figure~\ref{isotonormalform} show a Legendrian isotopy of one of the lower left strands. The lower left diagram indicates that the upper right diagram has moved one of the lower left strands to an upper right strand at the expense of changing the tangle $\mathcal{T}$. We can continue to get rid of all the lower left strands and similarly all of the lower right strands as well. This results in the middle diagram on the bottom row of Figure~\ref{isotonormalform}. Finally, one can isotopy the tangle $\mathcal{T}''$ out from the region bounded by $L$. This gives the claimed isotopy to the diagram on the left of Figure~\ref{normalform}.

\begin{figure}[htb]{
\begin{overpic}%[grid,tics=5] 
{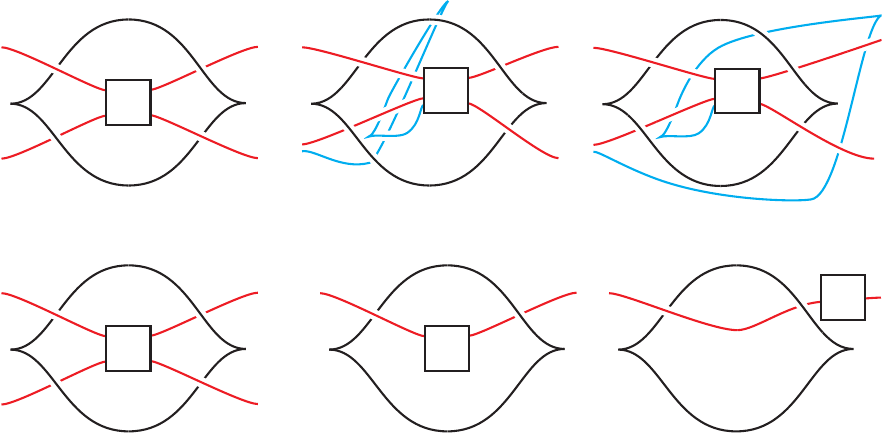}
\put(13.5, 36.5){$\mathcal{T}$}
\put(0, 32){$c$}
\put(0, 45){$a$}
\put(28.5, 32){$d$}
\put(28.5, 45){$b$}
\put(13.5, 8.75){$\mathcal{T'}$}
\put(-3, 4.25){$c-1$}
\put(0, 17.25){$a$}
\put(28.5, 4.25){$d$}
\put(25.5, 17.25){$b+1$}
\put(49.25, 8.75){$\mathcal{T''}$}
\put(94, 14.5){$\mathcal{T''}$}
\put(36, 17.25){$a+d$}
\put(69, 17.25){$a+d$}
\put(49.5, 38){$\mathcal{T}$}
\put(82.5, 38){$\mathcal{T}$}
\put(31, 34){$c-1$}
\put(34, 45){$a$}
\put(62.5, 32){$d$}
\put(62.5, 45){$b$}
\put(65, 34){$c-1$}
\put(68, 45){$a$}
\put(98, 32){$d$}
\put(99, 42){$b$}
\end{overpic}}
\caption{Isotopy to normalize the way strands run through $L$.}
\label{isotonormalform}
\end{figure}
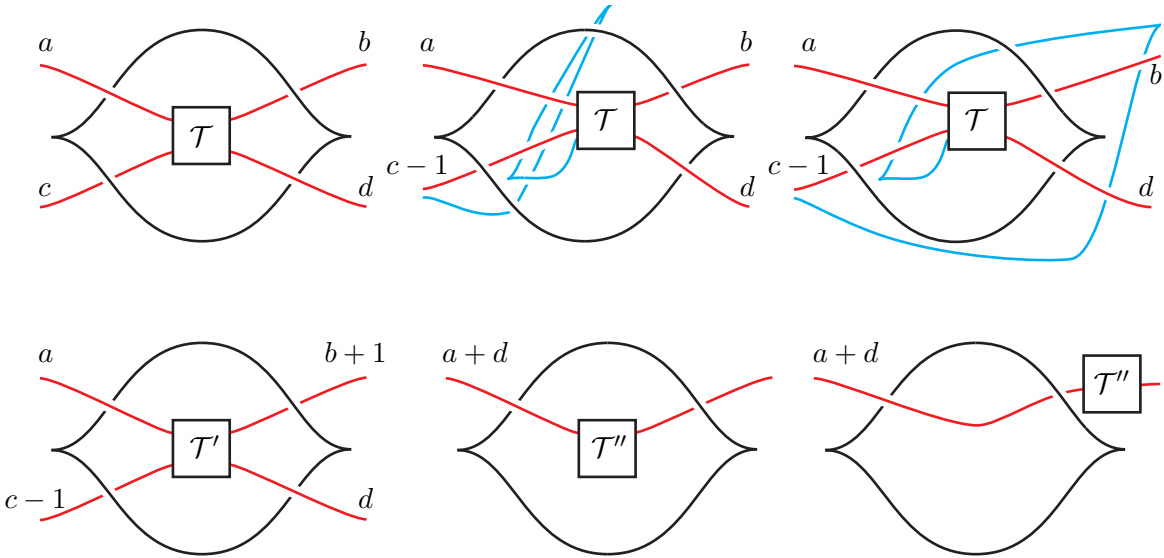

Smoothly doing contact $(1+1/1)$--surgery on $L$ (that is smooth $1$ surgery) is smoothly equivalent to replacing the left hand side of Figure~\ref{normalform} with the right hand side and changing the framings on the strands by subtracting their linking squared with $L$.

Now notice that if we realize the right hand side of Figure~\ref{normalform} by concatenating $n$ copies of either diagram in Figures~\ref{sandz} (where $n$ is the number of red strands in Figure~\ref{normalform}) then the Thurston-Bennequin invariant of each knot involved in Figure~\ref{normalform} is reduced by the linking squared with $L$. We claim that this is a Stein diagram for the result of $(2)$ contact surgery on $L$. To see this, recall that $(2$) contact surgery is effected by $(1)$ contact surgery on $L$ followed by $(-1)$ contact surgery on a once stabilized copy $L'$ of $L$ (that is $L'$ is obtained from $L$ by translating slightly up in the front diagram and then stabilizing), see \cite{DG:surgery}. Now if one handle slides the strands running through $L$ on the left hand side in Figure~\ref{normalform} over $L'$, as in the third row of Figure~25 in \cite{Avdek13}, one may cancel $L$ and $L'$ from the diagram, resulting in a link that is Legendrian isotopic to the one described above. 
\begin{figure}[htb]{
\begin{overpic}%[grid,tics=10] 
{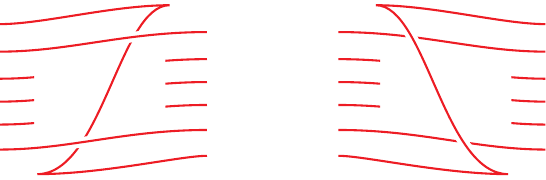}
%\put(74, 7.75){$-1$}
\end{overpic}}
\caption{Legendrian representations for negative twisting.}
\label{sandz}
\end{figure}

Notice that the diagram clearly describes an acyclic 4--manifolds and moreover the presentation for its fundamental group is the same as for the presentation for the fundamental group of $B^4$ given by the original diagram. Thus the 4--manifolds is contractible. 
\end{proof}

We now turn to the proof that connected sums of lens spaces can never have acyclic symplectic fillings, but first prove Proposition~\ref{prop} that says any contact structure on a lens space that is symplectically filled by a rational homology ball must be universally tight. 
\begin{proof}[Proof of Proposition~\ref{prop}]
Let $X$ be a rational homology ball symplectic filling of $L(p,q)$. We show the induced contact structure must be  the universally tight contact structure $\xi_{std}$. This will follow from unpacking recent work of Menke \cite{menke2018jsjtype} where he studies exact symplectic fillings of a contact $3$--manifold that contains a {\it mixed torus}. 

We start with the set-up. Honda \cite{Honda:classification1} and Giroux \cite{Giroux:classification} have classified tight contact structures on lens spaces. We review the statement of Honda in terms of the Farey tessellation. We use notation and terminology that is now standard, but see see \cite{Honda:classification1} for details. Consider a minimal path in the Farey graph that starts at $-p/q$ and moves counterclockwise to $0$. To each edge in this path, except for the first and last edge, assign a sign. Each such assignment gives a tight contact structure on $L(p,q)$ and each tight contact structures comes from such an assignment. If one assigns only $+$'s or only $-$'s to the edges then the contact structure is universally tight, and these two contact structures have the same underlying plane field, but with opposite orientations. We call this plane field (with either orientation) the the universally tight structure $\xi_{std}$ on $L(p,q)$. All the other contact structures are virtually overtwisted, that is they are tight structures on $L(p,q)$ but become overtwisted when pulled to some finite cover. The fact that at some point in the path describing a virtually overtwisted contact structure the sign must change is exactly the same as saying a Heegaard torus for $L(p,q)$ satisfies Menke's mixed torus condition. 

\begin{theorem}[Menke]\label{mixed}
Let $(Y, \xi)$ denote closed, co-oriented contact $3$--manifold and let $(W, \omega)$ be its strong (resp. exact) symplectic filling. If $(Y,\xi)$ contains a mixed torus $T$, then there exists a (possibly disconnected) symplectic manifold $(W', \omega')$ such that:

\begin{itemize}
\item $(W', \omega')$ is a strong (rep. exact) symplectic filling of its boundary $(Y',\xi')$.
\item $\partial W'$ is obtained from $\partial W$ by cutting along $T$ and gluing in two solid tori. 
\item $W$ can be recovered from $W'$ by symplectic round $1$--handle attachment.  

\end{itemize}  
\end{theorem} 

In our case we have $X$ filling $L(p,q)$. Suppose the contact structure on $L(p,q)$ is virtually overtwisted. The theorem above now gives a symplectic manifold $X'$ two which a round 1--handle can be attached to recover $X$; moreover, $\partial X'$ is a union of two lens spaces or $S^1\times S^2$. However, Menke's more detailed description of $\partial X'$ shows that $S^1\times S^2$ is not possible. We digress for a moment to see why this last statement is true. When one attaches a round 1--handle, on the level of the boundary, one cuts along the torus $T$ and then glues in two solid tori. Menke gives the following algorithm to determine the meridional slope for these tori. That $T$ is a mixed torus means there is a path in the Farey graph with three vertices having slope $r_1, r_2,$ and $r_3$, each is counterclockwise of the pervious one and there is an edge from $r_i$ to $r_{i+1}$ for $i=1,2$. The torus $T$ has slope $r_2$ and the signs on the edges are opposite. Now let $(r_3,r_1)$ denote slopes on the Farey graph that are (strictly) counterclockwise of $r_3$ and (strictly) clockwise of $r_1$. Any slope in $(r_3,r_1)$ with an edge to $r_2$ is a possible meridional slope for the glued in tori, and these are the only possible slopes. Now since our $r_i$ are between $-p/q$ and $0$ we note that if there was an edge from $r_2$ to $-p/q$ or $0$ then $r_2$ could not be part of a minimal path form $-p/q$ to $0$ that changed sign at $r_2$. Thus when we glue in the solid tori corresponding to the round 1--handle attachment, they will not have meridional slope $0$ or $-p/q$ and thus we cannot get $S^1\times S^2$ factors. 

The manifold $X'$ is either connected or disconnected. We notice that it cannot be connected because it is known that any contact structure on a lens space is planar \cite{Schoenenberger05}, and  Theorem~1.2 from \cite{Etnyre:planar} says any filling of a contact structure supported by a planar open book must have connected boundary. Thus we know that $X'$ is, in fact, disconnected. So $X'=X'_1\cup X'_2$ with $\partial X'_i$ a lens space. The Mayer--Vietoris sequence for the the decomposition of $X'$ into $X'_1\cup X'_2$ (glued along an $S^1\times D^2$ in their boundaries) shows that the first homology of $X_1'$ or $X_2'$ has rank 1, while both of their higher Betti numbers are 0. But now the long exact sequence for the pair $(X'_i,\partial X'_i)$ implies that $b_1$ must be 0 for both the $X_i'$. This contradiction shows that a symplectic manifold which is a rational homology ball and with convex boundary $L(p,q)$ must necessarily induce the universally tight contact structure on the boundary. 
%Checking the long exact sequence for the pair $(X'_i,\partial X'_i)$ one sees that $H_1(X'_i;\mathbb{Q})=\{0\}$. To recover $X_1$, we attach a round 1--handle. This is the same as attaching a 1--handle (that connects the components of $X'$) and a 2--handle. After attaching the 1--handle it is clear that $H_1$ of the result is still trivial. Now attaching the 2--handle can easily be seen to add 1 to the second Betti number, but this contradicts the fact that $X_1$ is a rational homology ball. Thus a strong exact symplectic manifold which is rational homology ball and with boundary $L(p,q)$ must necessarily induce the universally tight contact structure on the boundary. This completes the proof.  
\end{proof}

\begin{proof}[Proof of Theorem~\ref{main2}]
The statement about embeddings follows directly from the statement about symplectic fillings. To prove that result let 
 $X$ be an exact symplectic filling of $L(p,q)\# L(p,p-q)$ that is also a rational homology ball.  Since any contact structure on $L(p,q)\# L(p,p-q)$ is supported by a planar open book, \cite{Schoenenberger05}, Wendl's result \cite{Wendl10} tells us that $X$ can be taken to be a Stein filling of $L(p,q)\# L(p,p-q)$. Observe that there is an embedded sphere in $\partial X$ as it is reducible. Eliashberg's result in \cite[Theorem~$16.7$]{CE} says that $X$ is obtained from another Stein manifold with convex boundary by attaching a 1--handle. Thus $X\cong X_1\natural X_2$ where $X_1$ and $X_2$ are exact symplectic manifolds with $\partial X_1=L(p,q)$ and $\partial X_2=L(p,p-q)$ or $X\cong X'\cup (\text{1--handle})$ where $X'$ is symplectic 4-manifold with boundary $\partial X'\cong L(p,q)\sqcup L(p,p-q)$. 
 
As argued above in the proof of Proposition~\ref{prop} it is not possible to have $X'$ with disconnected boundary being lens spaces
%Recall that any contact structure on a lens space is planar \cite{Schoenenberger05}, and  Theorem~1.2 from \cite{Etnyre:planar} says any filling of a contact structure supported by a planar open book must have connected binding. Thus it is not possible to have $X'$ above 
and we must be in the case $X\cong X_1\natural X_2$; moreover, since $X$  is a rational homology balls, so are the $X_i$. Moreover, since $X_1$ and $X_2$ are symplectic filling of their boundaries, they induce tight contact structures on $L(p,q)$ and $L(p,p-q)$, respectively. 

Proposition~\ref{prop} says that these tight contact structures must be, which are unique up to changing orientation, universally tight contact structures $\xi_{std}$ on $L(p,q)$ and $\xi'_{std}$ on $L(p,p-q)$. Thus we have that $X_1$ and  $X_2$ are rational homology balls, and are exact symplectic fillings of $(L(p,q), \xi_{std})$, and $(L(p,p-q), \xi'_{std})$, respectively. In \cite[Corollary~1.2(d)]{Lisca08} Lisca classified all such fillings. According to Lisca's classification, symplectic rational homology ball fillings of  $(L(p,q), \xi_{std})$ are possible exactly when  $(p,q)=(m^2, mh-1)$ for some $m$ and $h$ co-prime natural numbers, and similarly for $(L(p,p-q), \xi'_{std})$ exactly when $(p,p-q)=(m^2, mk-1)$ for $m$ and $k$ co-prime natural numbers. Now simple calculation shows that, the only possible value for $m$ satisfying these equations is $m=2$. In particular, we get that $p=4$, but then we must have $\{q, p-q\}=\{1,3\}$, and 3 cannot be written as $2k-1$, for $k$ co-prime to $2$. Thus there is no such $X$. 
\end{proof}

\bibliography{references}

\providecommand{\bysame}{\leavevmode\hbox to3em{\hrulefill}\thinspace}
\providecommand{\MR}{\relax\ifhmode\unskip\space\fi MR }
\providecommand{\MRhref}[2]{%
  \href{http://www.ams.org/mathscinet-getitem?mr=#1}{#2}
}
\providecommand{\href}[2]{#2}
\begin{thebibliography}{10}

\bibitem{AKi}
Selman Akbulut and Robion Kirby, \emph{Mazur manifolds.}, The Michigan
  Mathematical Journal \textbf{26} (1979), no.~3, 259--284.

\bibitem{Avdek13}
Russell Avdek, \emph{{Contact surgery and supporting open books}}, Algebraic \&
  Geometric Topology \textbf{13} (2013), no.~3, 1613 -- 1660.

\bibitem{CH}
Andrew Casson and John Harer, \emph{Some homology lens spaces which bound
  rational homology balls}, Pacific Journal of Mathematics \textbf{96} (1981),
  no.~1, 23--36.

\bibitem{Cha07}
Jae~Choon Cha, \emph{The structure of the rational concordance group of knots},
  Mem. Amer. Math. Soc. \textbf{189} (2007), no.~885, x+95. \MR{2343079}

\bibitem{CE}
K.~Cieliebak and Y.~Eliashberg, \emph{From {S}tein to {W}einstein and {B}ack --
  {S}ymplectic {G}eometry of {A}ffine {C}omplex {M}anifolds}, Colloquium
  Publications, vol.~59, Amer. Math. Soc., 2012.

\bibitem{CET}
James Conway, John~B Etnyre, and B{\"u}lent Tosun, \emph{{Symplectic Fillings,
  Contact Surgeries, and Lagrangian Disks}}, International Mathematics Research
  Notices (2019), rny291.

\bibitem{DG:surgery}
Fan Ding and Hansj{\"o}rg Geiges, \emph{A {L}egendrian {S}urgery {P}resentation
  of {C}ontact 3-{M}anifolds}, Math. Proc. Cambridge Philos. Soc. \textbf{136}
  (2004), 583--598.

\bibitem{Donald}
Andrew Donald, \emph{Embedding {S}eifert manifolds in {$S^4$}}, Trans. Amer.
  Math. Soc. \textbf{367} (2015), no.~1, 559--595. \MR{3271270}

\bibitem{Epstein}
D.~B.~A. Epstein, \emph{Embedding punctured manifolds}, Proc. Amer. Math. Soc.
  \textbf{16} (1965), 175--176. \MR{208606}

\bibitem{Etnyre:planar}
John~B. Etnyre, \emph{Planar {O}pen {B}ook {D}ecompositions and {C}ontact
  {S}tructures}, IMRN \textbf{79} (2004), 4255--4267.

\bibitem{Fickle}
Henry~Clay Fickle, \emph{Knots, {${\bf Z}$}-homology {$3$}-spheres and
  contractible {$4$}-manifolds}, Houston J. Math. \textbf{10} (1984), no.~4,
  467--493. \MR{774711}

\bibitem{FS:lensspace}
Ronald Fintushel and Ronald Stern, \emph{Rational homology cobordisms of
  spherical space forms}, Topology \textbf{26} (1987), no.~3, 385--393.
  \MR{899056}

\bibitem{FintushelStern84}
Ronald Fintushel and Ronald~J. Stern, \emph{A {$\mu$}-invariant one homology
  {$3$}-sphere that bounds an orientable rational ball}, Four-manifold theory
  ({D}urham, {N}.{H}., 1982), Contemp. Math., vol.~35, Amer. Math. Soc.,
  Providence, RI, 1984, pp.~265--268. \MR{780582}

\bibitem{FS:R}
\bysame, \emph{Pseudofree orbifolds}, Ann. of Math. (2) \textbf{122} (1985),
  no.~2, 335--364. \MR{808222}

\bibitem{Fossati19pre}
Edoardo Fossati, \emph{Topological constraints for stein fillings of tight
  structures on lens spaces}, 2019.

\bibitem{Freedman:4manifolds}
Michael~Hartley Freedman, \emph{The topology of four-dimensional manifolds}, J.
  Differential Geometry \textbf{17} (1982), no.~3, 357--453. \MR{679066}

\bibitem{GilmerLivingston}
Patrick~M. Gilmer and Charles Livingston, \emph{On embedding {$3$}-manifolds in
  {$4$}-space}, Topology \textbf{22} (1983), no.~3, 241--252. \MR{710099}

\bibitem{Giroux:classification}
Emmanuel Giroux, \emph{Structures de contact en dimension trois et bifurcations
  des feuilletages de surfaces}, Invent. Math. \textbf{141} (2000), no.~3,
  615--689. \MR{1779622}

\bibitem{GollaStarkston19pre}
Marco Golla and Laura Starkston, \emph{The symplectic isotopy problem for
  rational cuspidal curves}, 2019.

\bibitem{Gompf}
Robert~E. Gompf, \emph{Handlebody {C}onstruction of {S}tein {S}urfaces}, Ann.
  Math. \textbf{148} (1998), 619--693.

\bibitem{Hantzsche}
W.~Hantzsche, \emph{Einlagerung von {M}annigfaltigkeiten in euklidische
  {R}\"{a}ume}, Math. Z. \textbf{43} (1938), no.~1, 38--58. \MR{1545714}

\bibitem{Hirsch:embedding}
Morris~W. Hirsch, \emph{On imbedding differentiable manifolds in euclidean
  space}, Ann. of Math. (2) \textbf{73} (1961), 566--571. \MR{124915}

\bibitem{Honda:classification1}
Ko~Honda, \emph{On the {C}lassification of {T}ight {C}ontact {S}tructures {I}},
  Geom. Topol. \textbf{4} (2000), 309--368.

\bibitem{Kawauchi09}
Akio Kawauchi, \emph{Rational-slice knots via strongly negative-amphicheiral
  knots}, Commun. Math. Res. \textbf{25} (2009), no.~2, 177--192. \MR{2554510}

\bibitem{Kirby:problemlist}
Rob Kirby, \emph{Problems in low dimensional manifold theory}, Algebraic and
  geometric topology ({P}roc. {S}ympos. {P}ure {M}ath., {S}tanford {U}niv.,
  {S}tanford, {C}alif., 1976), {P}art 2, Proc. Sympos. Pure Math., XXXII, Amer.
  Math. Soc., Providence, R.I., 1978, pp.~273--312. \MR{520548}

\bibitem{Lisca08}
Paolo Lisca, \emph{On symplectic fillings of lens spaces}, Trans. Amer. Math.
  Soc. \textbf{360} (2008), no.~2, 765--799. \MR{2346471}

\bibitem{Manolescu:T}
Ciprian Manolescu, \emph{Pin(2)-equivariant {S}eiberg-{W}itten {F}loer homology
  and the triangulation conjecture}, J. Amer. Math. Soc. \textbf{29} (2016),
  no.~1, 147--176. \MR{3402697}

\bibitem{MT:pseudoconvex}
Thomas~E. Mark and B\"{u}lent Tosun, \emph{Obstructing pseudoconvex embeddings
  and contractible {S}tein fillings for {B}rieskorn spheres}, Adv. Math.
  \textbf{335} (2018), 878--895. \MR{3836681}

\bibitem{menke2018jsjtype}
Michael Menke, \emph{A jsj-type decomposition theorem for symplectic fillings},
  2018.

\bibitem{OS:grading}
Peter Ozsv{\'a}th and Zolt{\'a}n Szab{\'o}, \emph{Absolutely graded floer
  homologies and intersection forms for four-manifolds with boundary}.

\bibitem{Rohlin:invariant}
V.~A. Rohlin, \emph{New results in the theory of four-dimensional manifolds},
  Doklady Akad. Nauk SSSR (N.S.) \textbf{84} (1952), 221--224. \MR{0052101}

\bibitem{Rohlin:3manembedding}
\bysame, \emph{The embedding of non-orientable three-manifolds into
  five-dimensional {E}uclidean space}, Dokl. Akad. Nauk SSSR \textbf{160}
  (1965), 549--551. \MR{0184246}

\bibitem{Schoenenberger05}
Stephan Sch\"onenberger, \emph{Planar open books and symplectic fillings},
  Ph.D. thesis, University of Pennsylvania, 2005.

\bibitem{Wall:embedding}
C.~T.~C. Wall, \emph{All {$3$}-manifolds imbed in {$5$}-space}, Bull. Amer.
  Math. Soc. \textbf{71} (1965), 564--567. \MR{175139}

\bibitem{Wendl10}
Chris Wendl, \emph{Strongly fillable contact manifolds and {$J$}-holomorphic
  foliations}, Duke Math. J. \textbf{151} (2010), no.~3, 337--384. \MR{2605865
  (2011e:53147)}

\bibitem{Zeeman}
E.~C. Zeeman, \emph{Twisting spun knots}, Trans. Amer. Math. Soc. \textbf{115}
  (1965), 471--495. \MR{195085}

\end{thebibliography}
\bibliographystyle{amsplain}

\end{document}